\documentclass[12pt]{amsart}

\usepackage{amsmath,amsfonts,amssymb,amsthm}
\usepackage{pinlabel}
\usepackage{graphicx}
\usepackage{mathtools}
\usepackage[all]{xy}
\usepackage{hyperref}
\usepackage[usenames,dvipsnames]{color}
\usepackage{color}
\xyoption{dvips}

\numberwithin{equation}{section}

\newtheorem{thm}{Theorem}[section]

\newtheorem{lem}[thm]{Lemma}

\theoremstyle{remark}

\theoremstyle{definition}

\theoremstyle{remark}
\newtheorem{remark}[thm]{Remark}
\theoremstyle{property}

\newcommand{\R}{\mathbb{R}}

\newcommand{\Z}{\mathbb{Z}}

\newcommand{\cH}{\mathcal H}
\newcommand{\cL}{\mathcal L}

\newcommand{\cA}{\mathcal A}
\newcommand{\cQ}{\mathcal Q}

\begin{document}

\title{A note on infinite number of exact Lagrangian fillings  for spherical spuns}

\author{Roman Golovko}

\begin{abstract}
In this short note we discuss high-dimensional examples of Legendrian submanifolds of the standard contact Euclidean space with an infinite number of exact Lagrangian fillings up to Hamiltonian isotopy.  They are obtained from the examples of Casals and Ng by applying to them the spherical spinning construction.
\end{abstract}

\address{Faculty of Mathematics and Physics, Charles University,
Prague, Czech Republic} \email{golovko@karlin.mff.cuni.cz}
\date{\today}
\thanks{}
\subjclass[2010]{Primary 53D12; Secondary 53D42}

\keywords{}

\maketitle

\section{Introduction}

Recently the question of existence of infinitely many exact Lagrangian fillings for Legendrian links has received a certain amount of interest. First it has been positively answered by Casals and Gao \cite{InfinitenumberfillCaGa}.
Later  the  works of An-Bae-Lee \cite{AnBaeLeeLagrfilfinnsft, AnBaeLeeLagrfilAFFsft}, Casals-Zaslow \cite{CasalsZaslowLegwea}, and  Gao-Shen-Weng \cite{GaoShenWenginfa, GaoShenWenginfb} have continued to develop various cluster and sheaf-theoretic
methods to detect infinitely many exact Lagrangian fillings for Legendrian links in the standard 3-dimensional contact vector space.

In \cite{CasalsNginfinitenumbeoffillings} Casals and Ng following the ideas of K\'{a}lm\'{a}n \cite{Kalmanconthom} have provided the first  series of Legendrian links in $(\R^3, \xi_{st})$
with the property that the Chekanov-Eliashberg algebra detects infinitely many exact Lagrangian fillings.
We show that spherical spinning construction applied to the examples of Casals and Ng leads to examples of Legendrian submanifolds of high-dimensional contact vector space with infinite number of pairwise Hamiltonian non-isotopic exact Lagrangian  fillings. More precisely, using Chekanov-Eliashberg algebras we prove the following:

\begin{thm}
\label{thminfinitenumberoffillings}
For a given $m\geq 1$ and $k_i\geq 2$, where $i=1,\dots,m$, there is a  Legendrian submanifold $\Lambda$ in the standard contact vector space  $\R^{2(k_1+\dots + k_m+1)+1}_{st}$ diffeomorphic to the disjoint union  of some number of $S^1\times S^{k_1}\times \cdots \times S^{k_m}$ which admits an infinite number of
exact Lagrangian fillings distinct up to Hamiltonian isotopy.
\end{thm}

\section{Background}
\subsection{Exact Lagrangian cobordisms}
We first recall the notions of exact Lagrangian cobordism and exact Lagrangian filling.

Let $\Lambda^{-}$ and $\Lambda^{+}$ be two Legendrian submanifolds of  the standard contact vector space $\R^{2n+1}_{st}:=(\R^{2n+1}, \alpha_{st}=dz-\sum y_i dx_i)$. We say that $\Lambda^{-}$ is exact Lagrangian cobordant to $\Lambda^{+}$ if there is a smooth cobordism $(L; \Lambda^{-},\Lambda^{+})$ and an exact Lagrangian embedding
$L\hookrightarrow S(\R^{2n+1}_{st})$, where $S(\R^{2n+1}_{st}):=(\R\times \R^{2n+1},  d(e^t\alpha_{st}))$ and $t$ is the coordinate on the first $\R$-factor, satisfying the following conditions:
\begin{itemize}
\item $L|_{(-\infty, -T)\times \R^{2n+1}_{st}}= (-\infty, -T)\times \Lambda^{-}$ and
$L|_{(T,\infty)\times \R^{2n+1}_{st}}= (T,\infty)\times \Lambda^{+}$ for some $T\gg 0$,
\item $L^{c}:=L|_{[-T,T]\times \R^{2n+1}_{st}}$ is compact.
\item There exists $f:L\to \R$ such that
$e^{t}\alpha_{st}|_{L} = df$ and $f|_{(-\infty,-T)\times \Lambda^{-}}$, $f|_{(T,\infty)\times \Lambda^{+}}$ are constant functions.
\end{itemize}
If $L$ is an exact
Lagrangian cobordism with empty negative end and whose positive end is equal to $\Lambda$, then we say that $L$ is an exact Lagrangian filling of $\Lambda$.

Given a closed, spin Legendrian  submanifold $\Lambda\subset \R^{2n+1}_{st}$
and an exact Lagrangian filling  $L\subset S(\R^{2n+1})$ of $\Lambda$,
following Ekholm, Honda, and K\'{a}lm\'{a}n \cite{EkhomHonaKalmancobordisms} observe that  $L$ induces an augmentation of the Chekanov-Eliashberg algebra  $(\cA(\Lambda), \Z_2[H_1(L)])$ onto  $\Z_2[H_1(L)]$. Here following the observation of Karlsson from \cite[Section 2.2]{KarlssonCorOrientCobordisms} we note that even though the original results of Ekholm, Honda, and K\'{a}lm\'{a}n from \cite[Section 3.5]{EkhomHonaKalmancobordisms} are formulated for $n=1$, tracing their proofs one sees that they can be
extended word-by-word to arbitrary $n$. Finally, following the work of Karlsson \cite{KarlssonCorOrientCobordisms}, one can extend the augmentation to be onto  $\Z[H_1(L)]$.

\begin{thm}[\cite{EkhomHonaKalmancobordisms, KarlssonCorOrientCobordisms}]
Let $L$ be a spin Maslov number $0$  exact Lagrangian filling  of closed spin Legendrian $\Lambda \subset \R^{2n+1}_{\xi_{st}}$ with Maslov number $0$.
Then $L$ induces an augmentation
$$
\varepsilon_L : (\cA(\Lambda), \Z[H_1(L)]) \to \Z[H_1(L)]
$$
where $\Z[H_1(L)]$ lies entirely in grading $0$.
 In addition, if $L$ and $L'$ are exact Lagrangian fillings of $\Lambda$ which are isotopic through exact Lagrangian fillings of $\Lambda$, then there is a DGA homotopy between corresponding augmentations $\varepsilon_L$ and $\varepsilon_{L'}$.
\end{thm}

All Legendrian submanifolds that we consider will have Maslov number $0$.
All the fillings we consider in this paper will have Maslov number $0$. In addition, as we will see for all Legendrian submanifolds $\Lambda$ in this paper Chekanov-Eliashberg algebra $\cA(\Lambda)$ will be supported entirely
in non-negative degree.

Following \cite{CasalsNginfinitenumbeoffillings} observe that in these settings, two DGA maps $$(\cA(\Lambda), \Z[H_1(L)])\to \Z[H_1(L)]$$ are DGA
homotopic if and only if they are equal, and hence if two fillings $L$, $L'$ produce augmentations to
$\Z[H_1(L)]$ that are distinct (under all isomorphisms identifying $H_1(L)$ and $H_1(L'))$, then $L$,$L'$ are not Hamiltonian isotopic.

\subsection{The spherical spinning construction}\label{sec:front-spinn-constr}
The front spinning construction has been introduced by Ekholm, Etnyre and Sullivan in \cite{EkholmEtnyreSullivannoniso}. From a given  Legendrian submanifold $\Lambda \subset \R^{2n+1}_{st}$ it produces  the Legendrian embedding of $\Lambda \times S^1$ inside $\R^{2n+3}_{st}$. This construction admits several extensions, see \cite{BourgeoisSabloffTraynorLagcobgenfam, DRG2021, EkholmKalmanisoLegknto, LambertColeLegPr, RutherfordSulivanfunctLCH, SabloffSullivanfamLegsub}. In this paper we will consider the spherical spinning construction described by the author in \cite{GolovkoSphericalSpinning}  which produces a Legendrian embedding of $\Lambda \times S^m$ inside $\R^{2(n+m)+1}_{st}$.
Spherical spinning, as shown in \cite{GolovkoSphericalSpinning}, can be extended to exact Lagrangian cobordisms. Now we recall the details of these constructions.

Consider the following embedding $\R \times S^n \hookrightarrow \R^{n+1}$ given by $(s, p) \mapsto e^s p$. Observe that it
induces an embedding $$\R^n \times S^m = \R^{n-1} \times \R \times S^m \hookrightarrow \R^{n+m}$$ which admits a canonical extension to an embedding
$$\R^{2n} \times T^*S^m = T^*\R^n \times T^*S^m \hookrightarrow T^*\R^{n+m}=\R^{2(n+m)}$$
preserving the the tautological one-forms.

Then we take a product $\Lambda\times 0_{S^m}$, where $0_{S^m}$ denotes the zero section of  $T^*S^m$. $\Lambda\times 0_{S^m}$ becomes  a Legendrian submanifold
of the contactization of $\R^{2n}\times T^{\ast}S^{m}$, and hence provides an embedding of $\Lambda \times S^m$ into the contactization of $\R^{2(n+m)}$.
This Legendrian embedding of $\Lambda \times S^m$ into $\R^{2(n+m)+1}_{st}$  is called the $S^m$-spun of $\Lambda$ (or just the spherical spun of $\Lambda$) and will be denoted by
$\Sigma_{S^m}\Lambda$.

We then observe that $S(\R^{2n+1}_{st})$  is symplectomorphic to the standard symplectic vector space $\R^{2(n+1)}$. For an exact Lagrangian cobordism $L$ from $\Lambda^-$ to $\Lambda^+$ in $S(\R^{2n+1}_{st})$, the image of the exact Lagrangian submanifold $L \times 0_{S^m} \subset \R^{2(n+1)} \times T^*S^m$ under the above embedding can be seen as an exact Lagrangian cobordism from $\Sigma_{S^m}\Lambda^-$ to $\Sigma_{S^m}\Lambda^+$ inside $S(\R^{2(n+m)+1}_{st})$.
This exact Lagrangian cobordism from $\Sigma_{S^m}\Lambda^-$ to $\Sigma_{S^m}\Lambda^+$ is diffeomorphic to $L \times S^m$ and is called the $S^m$-spun of $L$ (or just the spherical spun of $L$) and will be denoted by $\Sigma_{S^m}L$.  For more details we refer the reader to \cite{GolovkoSphericalSpinning} (where we discuss the embedding which defines $\Sigma_{S^m} L$) and \cite[Section 2]{ChDRGhG}.

\section{Proof of Theorem
\ref{thminfinitenumberoffillings}}

We consider the class of Legendrian links described by Casals and Ng in \cite{CasalsNginfinitenumbeoffillings}  with the property that the Chekanov-Eliashberg algebra detects infinitely many exact Lagrangian fillings. We denote this class by $\cH$.

As observed by Casals and Ng in \cite{CasalsNginfinitenumbeoffillings} each element $\Lambda\in \cH$ has rotation number
$0$ on each component, and all of the fillings that Casals and Ng have constructed are composed of minimum
cobordisms and saddle cobordisms at Reeb chords with degree $0$. It follows that each of
these fillings has Maslov number $0$.
In addition, for all $\Lambda\in\cH$ all  Reeb chords lie in nonnegative degree.
In these settings,  as observed by Casals and Ng \cite{CasalsNginfinitenumbeoffillings} two augmentations of  $\cA(\Lambda)$ are DGA homotopic if and only if they are equal.  Hence,  if two of the constructed diffeomorphic fillings $L$ and  $L'$ induce augmentations of $\cA(\Lambda)$ into  $\Z[H_1(L)]$ that are distinct under all isomorphisms identifying $H_1(L)$ and $H_1(L')$, then $L$ and $L'$ are not Hamiltonian isotopic.

Now we take $\Lambda \in \cH$ and consider $\Sigma_{S^k}\Lambda$, where $k\geq 2$.

\begin{remark}
Since $\Sigma_{S^k} \Lambda$  is diffeomorphic to $\Lambda\times S^k$ and $\Sigma_{S^k} L$ is diffeomorphic to $L\times S^k$, $k\geq 2$,
 by K\"{u}nneth formula we get that $H_1(\Sigma_{S^k} \Lambda)\simeq H_1(\Lambda)$,  $H_1(\Sigma_{S^k} L)\simeq H_1(L)$, so  $\Z[H_1(\Sigma_{S^k} \Lambda)]\simeq \Z[H_1(\Lambda)]$ and $\Z[H_1(\Sigma_{S^k} L)]\simeq \Z[H_1(L)]$.
\end{remark}

Following the discussion in \cite[Section 3]{EstimatingReebChordsCharacteristicAlgebra} we choose the perturbation of  $\Sigma_{S^k} \Lambda$, we will still call it $\Sigma_{S^k} \Lambda$, for which there is a decomposition of Reeb chords
$\mathcal{Q}(\Sigma_{S^k} \Lambda)=\mathcal{Q}_N \sqcup \mathcal{Q}_S$, where there is a canonical bijection between
$\mathcal{Q}_S \simeq \mathcal{Q}(\Lambda)$ which preserves the index of the chords, and there is also a canonical bijection $\mathcal{Q}_N \simeq \mathcal{Q}(\Lambda)$, which increases the grading by $k$.
Following the proof of \cite[Theorem 3.1]{EstimatingReebChordsCharacteristicAlgebra} we note that grading-preserving
bijection between the Reeb chords $\mathcal{Q}_S$  and the
Reeb chords on $\Lambda$ leads to the DGA inclusion
\begin{align}
\label{inclDGA}
i: \mathcal{A}(\Lambda) \hookrightarrow \mathcal{A}(\Sigma_{S^k} \Lambda),
\end{align}
which can be left inverted by the DGA map
$$\pi: \mathcal{A}(\Sigma_{S^k} \Lambda) \to \mathcal{A}(\Sigma_{S^k} \Lambda) /\langle \mathcal{Q}_N\rangle =\mathcal{A}(\Lambda) $$
which is given by taking the quotient of $\mathcal A(\Sigma_{S^k} \Lambda)$ with the two-sided ideal generated by $\mathcal{Q}_N$.

Then to each element $c\in \cQ(\Lambda)$ we associate two elements $c_S\in \mathcal{Q}_S$ and $c_N\in \mathcal{Q}_N$ obtained by the canonical bijections described above.

From the fact that the Chekanov-Eliashberg algebra $\mathcal A(\Lambda)$ is supported entirely in nonnegative degree
and the description of Reeb chords of the perturbed $\Sigma_{S^k} \Lambda$ it follows that $\cA(\Sigma_{S^k} \Lambda)$ is supported entirely in nonnegative degree. In addition, since $\Lambda\in \cH$, from the construction of $\Sigma_{S^k}\Lambda$ it follows that the Maslov number of $\Sigma_{S^k}\Lambda$ is zero.

\begin{remark}
\label{rem:one-to-one_aug}
Let $Aug(\Lambda)$ and $Aug(\Sigma_{S^k}\Lambda)$ denote the sets of graded augmentations of $\cA(\Lambda)$ and $\cA(\Sigma_{S^k}\Lambda)$, respectively. Then  it follows
that for every element $c \in \cQ(\Lambda)$ of degree $0$, $i(c)=c_S\in  \mathcal{Q}_S$  and $\pi(c_S)=c$. From this, together with the fact that all Reeb chords in $\mathcal{Q}_N$ are of non-zero degree, we see that $i^{\ast}: Aug(\Sigma_{S^k}\Lambda)\to Aug(\Lambda)$
and $\pi^{\ast}: Aug(\Lambda)\to Aug(\Sigma_{S^k}\Lambda)$ are inverse maps and hence there is a one-to-one correspondence between graded augmentations  of $\cA(\Lambda)$ and $\cA(\Sigma_{S^k}\Lambda)$, which to a graded augmentation $\varepsilon$ on $\cA(\Lambda)$ associates a graded augmentation $\tilde{\varepsilon}$ of $\cA(\Sigma_{S^k}\Lambda)$
defined by $\tilde{\varepsilon}(c_S)=\varepsilon(c)$,  $\tilde{\varepsilon}(c_N)=0$ for $c\in \cQ(\Lambda)$.
\end{remark}

Let $\cL$ denote the infinite set of exact Lagrangian fillings of $\Lambda$ with Maslov number zero which are diffeomorphic, but not Hamiltonian isotopic constructed by Casals and Ng.
First, we observe that by the discussion in \cite{GolovkoSphericalSpinning} if $L$ is an exact Lagrangian filling of $\Lambda$, then $\Sigma_{S^k} L$ is an exact Lagrangian filling of $\Sigma_{S^k} \Lambda$ diffeomorphic to $L\times S^k$. In addition, since the Maslov number of $L$ is zero, the Maslov number of $\Sigma_{S^k} L$ is zero.

We define
$$\Sigma_{S^k} \cL:=\{ \Sigma_{S^k} L \ |\  L\in \cL \}.$$
Since all elements of $\cL$ are pairwise diffeomorphic, all elements of $\Sigma_{S^k} \cL$ are pairwise diffeomorphic.

It remains to show that all elements of $\Sigma_{S^k} \cL$ are pairwise Hamiltonian non-isotopic.
Let $L$ and $L'$ be two different elements of $\cL$. We take $\Sigma_{S^k} L$ and $\Sigma_{S^k} L'$ and we want to show that  $\Sigma_{S^k} L$ and $\Sigma_{S^k} L'$ are not Hamiltonian isotopic. We argue by contradiction, i.e. we assume that $\Sigma_{S^k} L$ and $\Sigma_{S^k} L'$ are Hamiltonian isotopic.

Following Casals and Ng \cite{CasalsNginfinitenumbeoffillings} we observe that since all Reeb chords of $\Sigma_{S^k} \Lambda$ have nonnegative degrees, and $\Sigma_{S^k}\cL$ consists of fillings with Maslov number zero of a Maslov number zero Legendrian $\Sigma_{S^k}\Lambda$, if  two augmentations $$\varepsilon_{\Sigma_{S^k} L},\varepsilon'_{\Sigma_{S^k} L'}: (\cA(\Sigma_{S^k} \Lambda), \Z[H_1(\Sigma_{S^k} L)])\to \Z[H_1(\Sigma_{S^k} L)]$$ induced by exact Lagrangian fillings $\Sigma_{S^k} L, \Sigma_{S^k} L'\in \Sigma_{S^k}\cL$ are distinct under all isomomorphisms identifying $H_1(\Sigma_{S^k} L)$ and $H_1(\Sigma_{S^k} L')$, then the fillings $\Sigma_{S^k} L, \Sigma_{S^k} L'$  are not Hamiltonian isotopic.

Hence, since we assume that  $\Sigma_{S^k} L$ and $\Sigma_{S^k} L'$ are Hamiltonian isotopic, there is an isomorphism $\phi:H_1(\Sigma_{S^k} L)\to H_1(\Sigma_{S^k} L')$ which induces an isomorphism $\varphi: \Z [H_1(\Sigma_{S^k} L)] \to \Z [H_1(\Sigma_{S^k} L')]$
such that
\begin{align}
\label{spuneqaugm}
\varphi\circ \varepsilon_{\Sigma_{S^k} L} = \varepsilon_{\Sigma_{S^k} L'}.
 \end{align}

From Remark \ref{rem:one-to-one_aug} it follows that there is a one-to-one correspondence between $Aug(\Lambda)$ and $Aug(\Sigma_{S^k}\Lambda)$ given by $i^{\ast}(\tilde{\varepsilon})=\varepsilon$, $\pi^{\ast}(\varepsilon)=\tilde{\varepsilon}$ so that $\tilde{\varepsilon}(c_S)=\varepsilon(c)$, $\tilde{\varepsilon}(c_N)=0$ for $c\in \cQ(\Lambda)$.

Then we need the following property of augmentations induced by exact Lagrangian fillings:
\begin{lem}
\label{rem:inducedaugmentationgeometric}
If $\varepsilon$ is the augmentation induced by an exact Lagrangian filling $L\in \cL$ of $\Lambda$, then $\tilde{\varepsilon}$ is the augmentation induced by the exact Lagrangian filling $\Sigma_{S^k} L\in \Sigma_{S^k} \cL$ of $\Sigma_{S^k} \Lambda$.
\end{lem}

\begin{proof}
From the discussion in \cite[Section 2]{EkhomHonaKalmancobordisms} and \cite[Section 4]{PanRutherfordFunctLCH} it follows that
exact Lagrangian filling $L$ leads to the conical Legendrian cobordism $\hat{L}\subset J^1(\R_{>0}\times \R)$.

Then we discuss what is $\hat{L}$.
Recall that $S(\R^{3}_{st})\times \R_w$ is contactomorphic to $J^1(\R_{>0}\times \R)$ through the following contactomorphism:
\begin{equation} \label{eq:contactomorphism}
\begin{array}{rrcl}
\Psi:& \ S(\R^{3}_{st})\times \R_w& \to & J^1(\R_{> 0}\times \R)\\
& (t, x, y, z, w) & \to &\big( (e^t, x), (z, e^t y), e^tz+w\big).\\
\end{array}
\end{equation}
The Legendrian lift  of $L$ that we denote by $L^{\mathit{leg}}$ can then be mapped to a Legendrian $\Psi(L^{\mathit{leg}})$ in $J^1( \R_{>0}\times \R)$; denote it by $\hat{L}$.
For $t>T$, $L$ is cylindrical over $\Lambda$, and the primitive $f$ is a constant, say $f=A$.
Hence on $J^1((e^{T}, \infty)\times \R)$, $\hat{L}$ can be parametrized by
\begin{equation}\label{eq:para}
(s, \ x(\theta),\  z(\theta),\  sy(\theta), \ sz(\theta)-A)
\end{equation}
where $s=e^t$ and $\theta$ parametrizes $\Lambda$ in $\R^{3}_{st}$ through $(x(\theta), y(\theta), z(\theta))$.
That is, using the notation from \cite[Section 4.1.1]{PanRutherfordFunctLCH}  $\hat{L}$ agrees with $j^1(s\cdot \Lambda - A)$
 when $s>e^T$. In addition, clearly
$\hat{L} \cap J^1((0,e^{T}] \times \R)$ is compact.

Now to the conical Legendrian cobordism $\hat{L}$ we associate the Morse cobordism $\overline{L}_0$ obtained by modifying
$\hat{L}$ near its positive cylindrical ends.
$\overline{L}_0$ will have a Morse maximum at $s = s_M$ followed
by a Morse minimum end at $s = s_m$.
The precise construction of it goes the following way:
Let $s_m, s_M>0$ be such that there are small positive $\epsilon, \epsilon'$ satisfying
$s_M-3\epsilon>e^T$, $s_m-3\epsilon'>s_M$.
Then we define a positive Morse function $h:\R_{>0}\to \R$ such that
\item  $$h(s) =\begin{cases} s & \mbox{ if } s<s_M-2\epsilon,\\
B_M-(s-s_M)^2 & \mbox{ if } s_M-\epsilon<s<s_M+\epsilon',\\
B_m+(s-s_m)^2 & \mbox{ if } s>s_m-\epsilon',
\end{cases}$$
where $B_m, B_M>0$ are  some constants; $h>0$ on $(0, s_{M}+\epsilon)$, $h'> 0$ on $(0,s_M)$,
$h(s)$ has a unique local maximum at $s_M$, a unique local minimum at $s_m$, and no other critical points.
Define the Morse cobordism $\overline{L}_0$ to be the Legendrian surface in $J^1((0, s_M+\epsilon)\times \R)$ such that it
agrees with $\hat{L}$ in $J^1((0, e^T]\times \R)$,
and it agrees with $j^1( h(s) \cdot \Lambda-A)$ in $J^1([e^T, s_M+\epsilon)\times \R)$.

The schematic picture of it appears on Figure\ref{fig:MorseCobordisms}. From \cite[Lemma 7.1]{PanRutherfordFunctLCH}
it follows that  the set of Reeb chords of $\overline{L}_0$ has the form $\cQ(\overline{L}_0)=\cQ(\overline{L}_0)_m\cup \cQ(\overline{L}_0)_M$. $\cQ(\overline{L}_0)_m$ is in canonical grading-preserving one to one correspondence with the set of Reeb chords of $\Lambda$, $\cQ(\overline{L}_0)_M$ is in canonical one to one correspondence with the set of Reeb chords of $\Lambda$ with the grading shift by $1$, i.e. $|\hat{x}|=|x|+1$ for $\hat{x}\in \cQ(\overline{L}_0)_M$ and the corresponding $x\in \cQ(\Lambda)$.

\begin{figure}[ht]
\labellist
\pinlabel $e^T$ at 235 553
\pinlabel $s_{+}$ at 258 553
\pinlabel $s$ at 285 553
\pinlabel $e^T$ at 410 553
\pinlabel $s_M$ at 433 553
\pinlabel $s_m$ at 474 553
\pinlabel $s$ at 510 553
\endlabellist
  \centering
  \includegraphics[height=4.4cm]{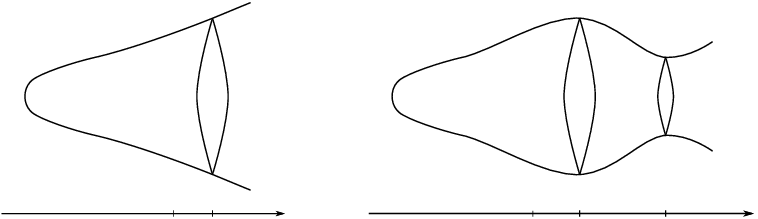}
  \vspace{5mm}
  \caption{Schematic pictures of $\hat{L}$ (left) and $\overline{L}_0$ (right)}
  \label{fig:MorseCobordisms}
\end{figure}
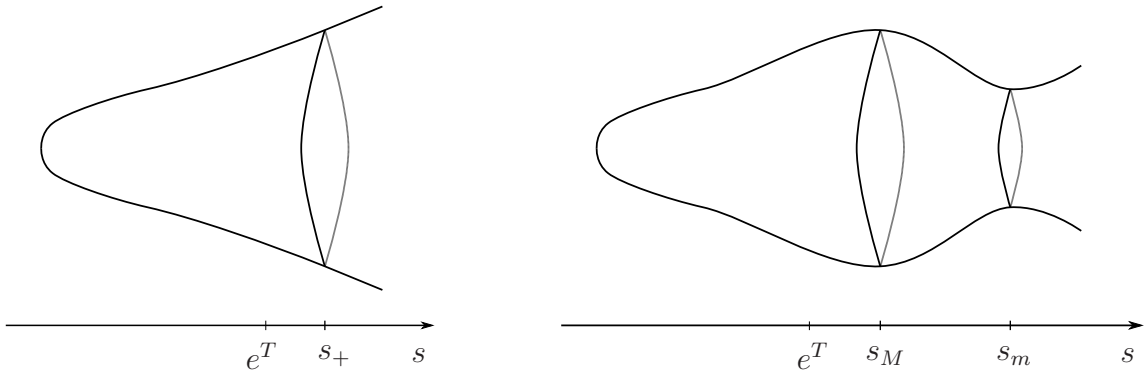

One defines the Chekanov-Eliashberg  DGA $\cA(\overline{L}_0)$
whose differential counts rigid gradient flow trees,
see\cite{PanRutherfordFunctLCH}. Recall that according to \cite{EkholmMorsefltrees2007}, counting gradient flow trees is an alternative way of defining differential of a Chekanov-Eliashberg algebra.
From \cite[Formula (7.1)]{PanRutherfordFunctLCH} it follows that
\begin{align}
\label{dbefspinn}
d(\hat{x})=\varepsilon_{L}(x)+n(d(\hat{x})),
\end{align}
where $n(d(\hat{x}))$ denotes the non-constant part of $d(\hat{x})$.

\begin{remark}
\label{coefficientsdiff}
Note that since Pan and Rutherford considered a class of all $1$-dimensional Legendrian submanifods in $J^1(M)$, the coefficients ring of the Chekanov-Eliashberg algebra used in the work of Pan and Rutherford \cite{PanRutherfordFunctLCH} is $\Z_2$, but for the Legendrian submanifolds and exact Lagrangian cobordisms that we consider the  proof of \cite[Formula (7.1)]{PanRutherfordFunctLCH} works the same way for $\Z[H_{1}(L)]$ coefficients ring.
\end{remark}

Then we apply the spherical spinning construction to $\overline{L}_0$ and get $\Sigma_{S^k} \overline{L}_0$. In addition, we consider exact Lagrangian cobordism  $\Sigma_{S^k} L$ and the corresponding Legendrian $\overline{\Sigma_{S^k} L}_0$.
Note that  $\Sigma_{S^k} \overline{L}_0$ coincides with $\overline{\Sigma_{S^k} L}_0$. This follows from the fact that contactomorphism \ref{eq:contactomorphism} preserves $x$-coordinate with respect to which we perform the spherical spinning construction in front projection, see \cite[Section 2]{GolovkoSphericalSpinning}, and in addition $s$-coordinate is not affected by the spherical spinning construction.
\begin{remark}
\label{cuspendgespun}
Following Ekholm \cite{EkholmMorsefltrees2007} observe that for two-dimensional Legendrians generically the only types of singularities of fronts are cusp-edges of codimension $1$ and swallow-tails of codimension $2$. Since we start with $1$-dimensional Legendrian submanifolds and their $2$-dimensional exact Lagrangian fillings to which we apply the spherical spinning construction, all singularities of the fronts that we have for spuns are as in the two-dimensional case. Therefore the flow trees result  of Ekholm \cite{EkholmMorsefltrees2007} relating flow trees and pseudoholomorphic disks in the definition of Chekanov-Eliashberg algebra is applicable for spuns that we consider.
\end{remark}

\begin{remark}
\label{formcounthd}
Following the proofs of Pan and Rutherford of \cite[Lemma 7.1]{PanRutherfordFunctLCH}  and \cite[Formula (7.1)]{PanRutherfordFunctLCH} that have been originally written for $1$-dimensional Legendrians and $2$-dimensional cobordisms, we note that the proofs remain correct in high dimensions in case when
the correspondence between flow trees and pseudoholomorphic disks proven by Ekholm in \cite{EkholmMorsefltrees2007} holds. Hence, using Remarks \ref{coefficientsdiff}, \ref{cuspendgespun},    we observe that the
analogue of Formula \ref{dbefspinn} holds for $\overline{\Sigma_{S^k} L}_0$.
\end{remark}

Again, using the discussion in  \cite[Section 3]{EstimatingReebChordsCharacteristicAlgebra} we have
an inclusion
\begin{align}
\label{inclofthsp}
i:\cA(\overline{L}_0)\hookrightarrow \cA(\Sigma_{S^k}\overline{L}_0).
\end{align}
The set of Reeb chords $\cQ(\Sigma_{S^k} \overline{L}_0)$ decomposes as $\cQ_N(\Sigma_{S^k} \overline{L}_0) \sqcup \cQ_S(\Sigma_{S^k} \overline{L}_0)$, where there is a canonical bijection between
$\mathcal{Q}_S(\Sigma_{S^k} \overline{L}_0) \simeq \mathcal{Q}(\overline{L}_0)$ which preserves the index of the chords, and there is also a canonical bijection $\mathcal{Q}_N(\Sigma_{S^k} \overline{L}_0) \simeq \mathcal{Q}(\overline{L}_0)$, which increases the grading by $k$. Reeb chords of
$\mathcal{Q}_S(\Sigma_{S^k} \overline{L}_0)$ will be denoted  by $x_S$
and Reeb chords of
$\mathcal{Q}_N(\Sigma_{S^k} \overline{L}_0)$ will be denoted by $x_N$.
Then using Remark \ref{formcounthd} we see that
\begin{align}
\label{daftspunce}
d(\hat{x}_S)=\varepsilon_{\Sigma_{S^k}  L}(x_S)+n(d(\hat{x}_S)).
\end{align}
From Formulas \ref{dbefspinn}, \ref{daftspunce} and the fact that there exists an inclusion \ref{inclofthsp} which is a DGA map it follows that
$$\varepsilon_{L}(x)=\varepsilon_{\Sigma_{S^k}  L}(x_S).$$ Using this together with the fact that the only Reeb chords of grading $0$ of $\Sigma_{S^k} \Lambda$ are in $\cQ_S(\Lambda)$ implies that $\varepsilon_{\Sigma_{S^k}  L}=\tilde{\varepsilon}$, where $\varepsilon=\varepsilon_L$.
\end{proof}

Then we apply $i^{\ast}$ to the Equation \ref{spuneqaugm} and get
$$ \varphi\circ i^{\ast}(\varepsilon_{\Sigma_{S^k} L})=\varphi\circ\varepsilon_{\Sigma_{S^k} L}\circ i=i^{\ast}(\varphi\circ \varepsilon_{\Sigma_{S^k} L}) =  i^{\ast}(\varepsilon_{\Sigma_{S^k} L'})$$
which using Lemma \ref{rem:inducedaugmentationgeometric} transforms to
\begin{align}
\label{eq:chareqbefspun}
\varphi\circ \varepsilon_{L}= \varepsilon_{L'}.
\end{align}
Equation \ref{eq:chareqbefspun} contradicts the fact that $L$ and $L'$ are two different elements of $\mathcal L$.
Thus, $\Sigma_{S^k} L$ is not Hamiltonian isotopic to  $\Sigma_{S^k} L'$.

Then we apply the spherical spinning construction $\Sigma_{S^{l}}$, $l>2$, to $\Sigma_{S^k} \Lambda$, and using the same steps we see that $\Sigma_{S^l} \Sigma_{S^k} \Lambda$ admits an infinite number of  Hamiltonian non-isotopic exact Lagrangian fillings that are distinguished by the Chekanov-Eliashberg algebra. We repeat this process as many times as we want. Finally, we see that  $\Sigma_{S^{k_m}} \dots \Sigma_{S^{k_1}}\Lambda$, which is  diffeomorphic to the disjoint union of some number of $S^1\times S^{k_1}\times \dots \times S^{k_m}$,
admits an infinite number of exact Lagrangian fillings up to Hamiltonian isotopy. This finishes the proof of Theorem \ref{thminfinitenumberoffillings}.

\section*{Acknowledgements}
The author is deeply grateful to Georgios Dimitroglou Rizell,  Cecilia Karlsson, Lenny Ng,
Yu Pan, and the referees of the previous version of this paper for very helpful discussions. The author is supported by the GA\v{C}R EXPRO Grant 19-28628X.

{}

\end{document}